\documentclass[10pt]{article}
\usepackage[affil-it]{authblk}
\usepackage{amsthm, amsmath,amstext,amscd, amssymb,txfonts, epsfig, psfrag, color, multicol, graphicx}
\usepackage[normalem]{ulem}
\usepackage[colorlinks=true, pdfstartview=FitV, linkcolor=blue, citecolor=blue, urlcolor=blue]{hyperref}

\newtheorem{thm}{Theorem}[section]
\newtheorem{cor}[thm]{Corollary}
\newtheorem{corollary}[thm]{Corollary}
\newtheorem{lemma}[thm]{Lemma}

\newtheorem{prop}[thm]{Proposition}

\theoremstyle{definition}

\newtheorem{Examples}[thm]{Examples}
\newtheorem{Questions}[thm]{Questions}
\newtheorem{Open questions}[thm]{Open questions}
\newtheorem{Open question}[thm]{Open question}
\newtheorem{Open problems}[thm]{Open problems}
\newtheorem{Open problem}[thm]{Open problem}

\definecolor{dmagenta}{rgb}{.5,0,.5} 
\definecolor{dred}{rgb}{.5,0,0} 
\definecolor{dgreen}{rgb}{0,.5,0} 
\definecolor{dblue}{rgb}{0,0,0.5} 
\definecolor{black}{rgb}{0,0,0} 
\definecolor{vdgreen}{rgb}{0,.3,0} 
\definecolor{vdred}{rgb}{.3,0,0} 
\definecolor{red}{rgb}{1,0,0} 
\definecolor{gray}{rgb}{.5,.5,.5}

\def\Bbb{\mathbb}
\def\bar{\overline}

\def\F{\Bbb{F}}
\def\Z{\Bbb{Z}}

\def\Q{\Bbb{Q}}

\def\tr{\hbox{\rm tr}}
\def\det{\textup{det}}

\def\Stab{\hbox{\rm Stab}}

\def\GL{\hbox{\rm GL}}
\def\SL{\hbox{\rm SL}}
\def\PGL{\hbox{\rm PGL}}
\def\PSL{\hbox{\rm PSL}}

\def\PL{\hbox{\rm PL}}
\def\ssm{\smallsetminus}

\def\ms{\medskip}

\def\onto{{\kern3pt\to\kern-8pt\to\kern3pt}}

\def\<{\langle}
\def\>{\rangle}
\def\|{{\ |\ }}

\newcommand{\set}[1]{\left\{#1\right\}}

\newcommand{\abs}[1]{\left|#1\right|}

\renewcommand{\ms}{\medskip}
\newcommand{\bs}{\bigskip}

\newcommand{\ds}{\displaystyle}

\begin{document}

\title{Graham Higman's lectures on januarials}

\author{Marston Conder%
 \thanks{\texttt{m.conder@auckland.ac.nz}, \href{http://www.math.auckland.ac.nz/~conder/}{http://www.math.auckland.ac.nz/$\sim$conder/}}}
\affil{Department of Mathematics, University of Auckland,  \\ Private Bag 92019, Auckland 1142,  New Zealand}

\author{Timothy Riley%
\thanks{\texttt{tim.riley@math.cornell.edu}, \href{http://www.math.cornell.edu/~riley/}{http://www.math.cornell.edu/$\sim$riley/}
}}
\affil{Department of Mathematics, 310 Malott Hall, \\  Cornell University, Ithaca, NY 14853, USA}
 
\date{}

\maketitle

\begin{abstract} 
\noindent This is an account of a series of lectures of Graham Higman on \emph{januarials},  
namely coset graphs for actions of triangle groups which become \emph{2-face maps} 
when embedded in orientable surfaces.
\end{abstract}


 \begin{tabular}{lll}
\hspace*{2mm} \parbox{36mm}{\vspace*{-28mm}\footnotesize{{\bf Spilt Milk} \\ We
that have done and thought, \rule{0mm}{4mm} \\ That have thought and done,  \\  
Must ramble, and thin out \\ Like milk spilt on a stone.}} & 
\epsfig{file=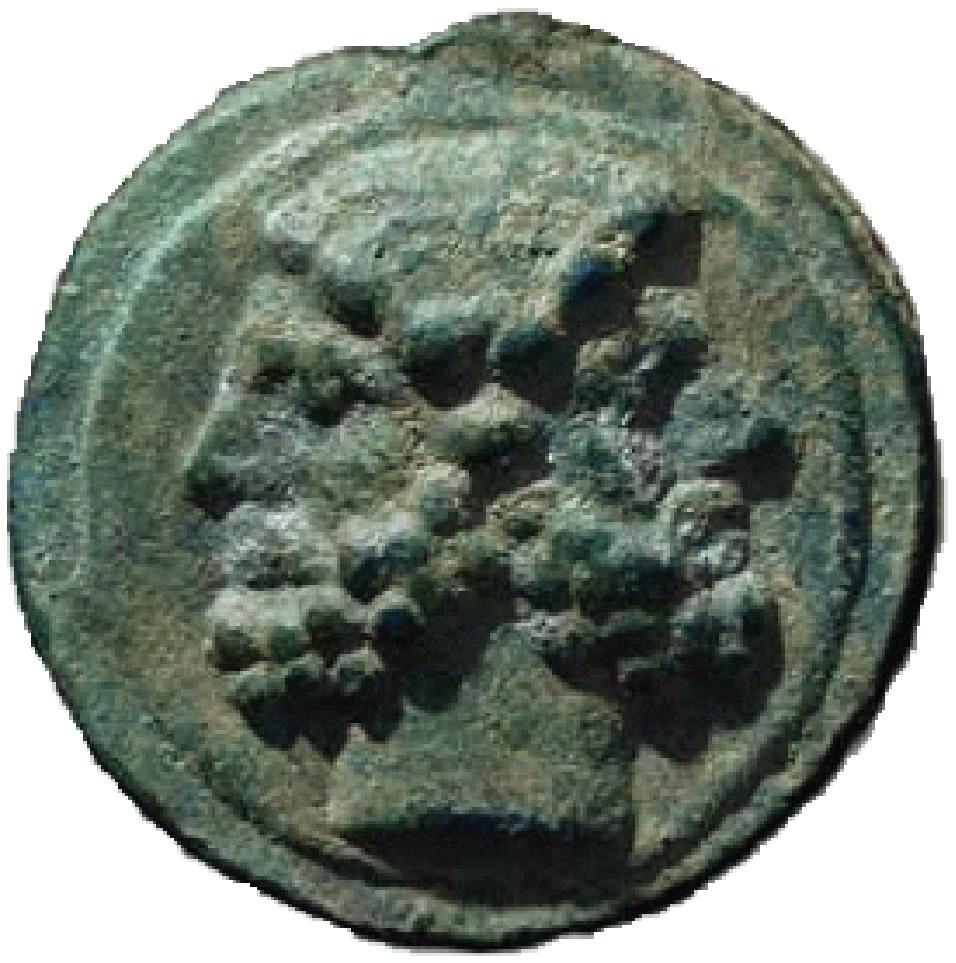,width=3cm}
\hspace{4mm} \parbox{6cm}{\vspace*{-30mm}\footnotesize{{\bf The
Nineteenth Century and After} \\ Though the great song return no more \rule{0mm}{4mm} \\ 
There's keen delight in what we have:   \\  The rattle of pebbles on the shore \\ 
Under the receding wave. }}
\end{tabular} 
\vspace*{-4mm}
\begin{flushright}
\footnotesize{from \emph{The Winding Stair and Other Poems}, W.B.Yeats, 1933}
\end{flushright}

\section{Preamble}

Graham Higman gave the lectures on which this article is based, in Oxford in 2001.  
They are likely to have been the final lectures he gave; he died in April 2008, at the age of 91.    
He introduced them with the quotes from W.B.~Yeats reproduced above, 
and described the work in preparing them as ``justifying my old age''  and ``keeping me relatively sane.'' 
The second author attended the lectures, and the first author remembers Higman's 
work on related topics some years earlier; this account is developed from recollections 
and from notes taken at the time.  
As such, any errors  are ours, and the presentation and the proofs offered may not 
be as Higman had in mind.  
At various points, and as indicated, we have extended Higman's treatment; 
we also include some of our own observations in an afterword.  

Januarials, which we will define in Section~\ref{defs}, are 2-complexes 
with two distinguished faces, that result from embedding coset graphs for the 
actions of triangle groups into compact orientable surfaces. 
They can be  viewed as being assembled  from two sub-surfaces (essentially those 
two distinguished faces); we give appropriate definitions and tools to explore the 
complexity of this assembly in Section~\ref{topological structure}.  
In Section~\ref{finding} we give sufficient conditions for actions of $\PSL(2,p)$ on 
projective lines to give rise to januarials.  This leads to a number of examples 
presented in Section~\ref{examples}.  Finally Section~\ref{afterword}, our afterword, 
contains some remarks on the coset graph appearing in Norman Blamey's 
1984 portrait of Higman,  and some further examples of januarials.  

 It appears that Higman's study of januarials was sparked by his work on {\em Hurwitz groups}, 
which are non-trivial finite quotients of the $(2,3,7)$-triangle group.  
Higman used coset diagrams to show that for all sufficiently large $n$, the alternating 
group $\textup{Alt}(n)$ is a Hurwitz group, and his work was taken further by the first 
author to determine exactly which $\textup{Alt}(n)$ are Hurwitz, in  \cite{Con-gensAnSn}. 

\ms

\emph{Funding.} This work was partially supported by the New Zealand Marsden Fund
[grant UOA1015] to MDEC; and the United States National Sciences Foundation [grant DMS--1101651] to TRR; and the Simons Foundation [collaboration grant 208567] to TRR.

\ms

\emph{Acknowledgements.} 
We are grateful to Martin Bridson, Harald Helfgott and Peter Neumann 
for encouragement and guidance in the writing of this account, and to Stephen Blamey and Sam~Howison (Chairman of the Oxford Mathematical Institute) for permission to reproduce Norman Blamey's portrait of Graham Higman in Figure~\ref{Higman portrait}.  
We also thank the referee for carefully reading this article and making some very helpful suggestions.

\section{Coset graphs, face maps,  januarials, and surfaces} \label{defs}

Suppose $S$ is a set  endowed with an action $s \mapsto s^g$ by a group $G$,  
and $A$ is a generating set for $G$.  
Define $\Gamma$ to be the graph with vertex set $S,$ and with an
oriented edge labelled by $a \in A$ (called an \emph{$a$-edge}) from vertex $u$ to
vertex $v$ whenever $u^a=v$.  
We will be concerned with situations where $G$ acts transitively on $S$, 
so that $\Gamma$ is connected.  
In that event we can identify $S$ with the right cosets $\set{\,Hg : g \in G\,}$ 
of the stabiliser $H = G_s = \Stab_G(s)$ of any particular $s \in S$, and for this reason,  
$\Gamma$  is known as a \emph{coset graph} or \emph{Schreier graph}, 
or sometimes \emph{coset diagram}, for the action of $G$ on $S$ with respect to $A$.  
When the action of $G$ on $S$ is also regular, we can identify $S$ with the underlying 
set of $G$, in which case $\Gamma$ is the \emph{Cayley graph} of $G$ with respect to $A$. 

Paths in the coset graph may be labelled with words on the generating set $A$ (which 
can be thought of as an alphabet).  
Suppose that a word $w$ on $A^{\pm 1}$  represents $g \in G$, and that $s \in S$.   
Let $\gamma$ be the path in $\Gamma$ obtained by concatenating the unique edge-paths 
in $\Gamma$ from $s^{g^i}$ to $s^{g^{i+1}}$, for each $i \in \Z$, along which one reads $w$.  
This  tours an orbit of $\langle g \rangle$ and is a (closed) circuit precisely when that orbit is finite.  
There is one such path for each orbit.

\begin{figure}[ht]
\psfrag{x2}{}
\psfrag{i}{\tiny{$\infty$}}
\psfrag{0}{\tiny{$0$}}
\psfrag{1}{\tiny{$1$}}
\psfrag{2}{\tiny{$2$}}
\psfrag{3}{\tiny{$3$}}
\psfrag{4}{\tiny{$4$}}
\psfrag{5}{\tiny{$5$}}
\psfrag{6}{\tiny{$6$}}
\psfrag{7}{\tiny{$7$}}
\psfrag{8}{\tiny{$8$}}
\psfrag{9}{\tiny{$9$}}
\psfrag{10}{\tiny{$10$}}
\psfrag{11}{\tiny{$11$}}
\psfrag{12}{\tiny{$12$}}
\psfrag{13}{\tiny{$13$}}
\psfrag{14}{\tiny{$14$}}
\psfrag{15}{\tiny{$15$}}
\psfrag{16}{\tiny{$16$}}
\centerline{\epsfig{file=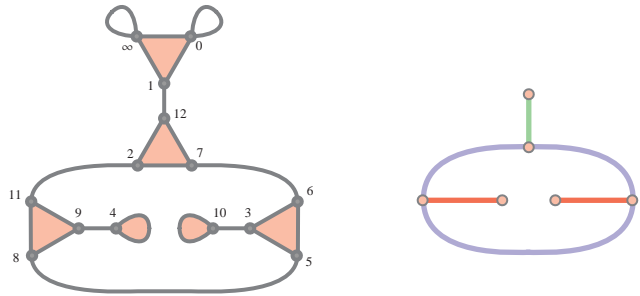}} 
\caption{A coset graph arising from an action of 
the group $\PSL(2,13)$ 
on $\F_{13} \cup \set{\infty}$, with $x : z \mapsto -z$ and $y : z \mapsto (z-1)/z$, 
and its companion graph.  
This results in a $3$-januarial of genus $0$ and simple type $( 1, 0, 0)$.} 
\label{13 associate}
\end{figure}

\begin{figure}[ht]
\psfrag{x2}{}
\psfrag{t}{$t: z \mapsto 1/z$} 
\psfrag{i}{\tiny{$\infty$}}
\psfrag{0}{\tiny{$0$}}
\psfrag{1}{\tiny{$1$}}
\psfrag{2}{\tiny{$2$}}
\psfrag{3}{\tiny{$3$}}
\psfrag{4}{\tiny{$4$}}
\psfrag{5}{\tiny{$5$}}
\psfrag{6}{\tiny{$6$}}
\psfrag{7}{\tiny{$7$}}
\psfrag{8}{\tiny{$8$}}
\psfrag{9}{\tiny{$9$}}
\psfrag{10}{\tiny{$10$}}
\psfrag{11}{\tiny{$11$}}
\psfrag{12}{\tiny{$12$}}
\psfrag{13}{\tiny{$13$}}
\psfrag{14}{\tiny{$14$}}
\psfrag{15}{\tiny{$15$}}
\psfrag{16}{\tiny{$16$}}
\centerline{\epsfig{file=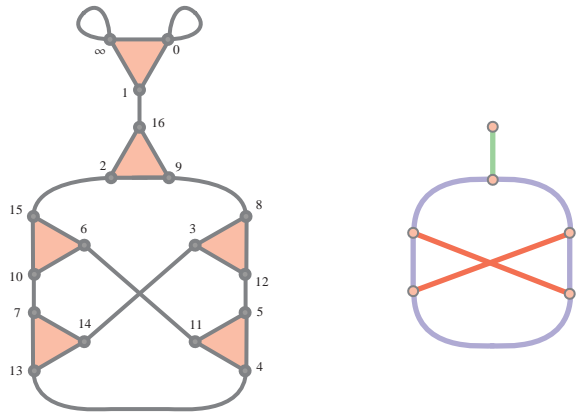}} 
\caption{A coset graph arising from an action of  
the group $\PSL(2,17)$ 
on $\F_{17} \cup \set{\infty}$, with $x : z \mapsto -z$ and $y : z \mapsto (z-1)/z$, 
and its companion graph.  
This results in a $3$-januarial of genus $1$ and simple type $( 1, 1, 0)$.} 
\label{17 associate}
\end{figure}

A \emph{map} is a 2-cell embedding of a connected (multi)graph 
in some closed surface, with its \emph{faces} (the components of 
the complement of the graph in the surface) being homeomorphic 
to open disks in $\mathbb{R}^2$.  Examples include triangulations and 
quadrangulations of the torus, and the Platonic solids (which may be 
viewed as highly symmetric maps on the sphere), with all vertices having the 
same valence and all faces having the same size. 

A \emph{januarial} is a special instance of a map constructed from embedding a coset graph for an action of the 
the $(2,k,l)$ \emph{triangle group} 
$$ 
\Delta(2,k,l) \ = \  \langle \, x,y \,  \mid \,  x^2 = y^k = (xy)^l=1 \, \rangle, 
$$   
with $A= \set{x, y}$. 
Because $x^2=1$, the $x$-edges in such a coset graph $\Gamma$ coming from non-trivial cycles of $x$ 
occur in pairs: whenever there is an  $x$-edge from $u$ to $v$, there is an $x$-edge from $v$ to $u$.  
We may identify each such pair, so as to leave  an \emph{unoriented} $x$-edge between $u$ and $v$.   
Then for each fixed point $s$ of $x$, we attach a 2-cell (which we will call 
an \emph{$x$-monogon}) along the $x$-edge which forms a loop at $s$.   
Similarly, for each orbit of $\langle y \rangle$, we attach a polygon (which we call a \emph{$y$-face}) 
along the path $\gamma$ given by $\langle y \rangle$ as described above.  
This gives a 2-complex, many examples of which appear in this article; see Figures~\ref{13 associate}, \ref{17 associate}, \ref{13 and 17 coset diagram},  \ref{PSL73 associate figure},  \ref{43_associate}, \ref{PSL31}, and  \ref{PSL31_associate}. These and similar figures can be displayed without labels on the edges, because we may 
shade the $y$-faces so that $y$-edges are identifiable as those in the boundaries of $y$-faces, 
while all the remaining edges are $x$-edges.  
We need not indicate orientations on the edges: the $x$-edges for the reason given 
above, and the $y$-edges because we may adopt a convention that all $y$-edges are 
oriented anti-clockwise around the corresponding $y$-faces.  
Note that the length (the number of sides) of each $y$-face divides $k$.

Next, attach a polygon (which we call an $xy$-face) around each orbit of $\langle xy \rangle$.  
As shown by the following lemma, the resulting 2-complex $J$ is homeomorphic to a
closed orientable surface.  
We may call the corresponding embedding of $\Gamma$ 
an  \emph{$m$-face map}, where $m$ is the number  of orbits of $\langle xy \rangle$.  

A \emph{januarial} (and more precisely, a \emph{$k$-januarial}) is the instance when $m=2$ and the orbits  
of $\langle xy \rangle$ have the same size $\abs{\,S}/2$.   
Two  examples of $3$-januarials are given  in Figure~\ref{sphere and torus}.

\begin{figure}[ht]
\psfrag{0}{\tiny{$0$}}
\psfrag{1}{\tiny{$1$}}
\psfrag{2}{\tiny{$2$}}
\psfrag{3}{\tiny{$3$}}
\psfrag{4}{\tiny{$4$}}
\psfrag{5}{\tiny{$5$}}
\psfrag{6}{\tiny{$6$}}
\psfrag{7}{\tiny{$7$}}
\psfrag{8}{\tiny{$8$}}
\psfrag{9}{\tiny{$9$}}
\psfrag{10}{\tiny{$10$}}
\psfrag{11}{\tiny{$11$}}
\psfrag{12}{\tiny{$12$}}
\psfrag{13}{\tiny{$13$}}
\psfrag{14}{\tiny{$14$}}
\psfrag{15}{\tiny{$15$}}
\psfrag{16}{\tiny{$16$}}
\psfrag{i}{\tiny{$\infty$}}
\centerline{\epsfig{file=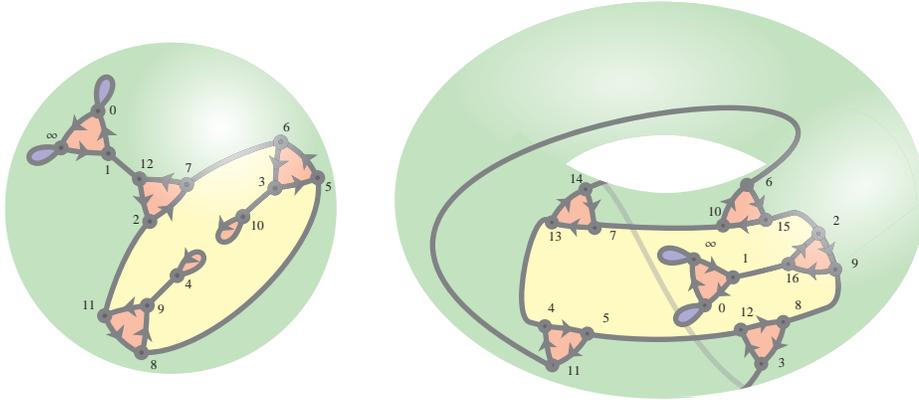}}
\caption{The $3$-januarials arising from the coset graphs in Figures~\ref{13 associate} and \ref{17 associate}. The  unoriented   edges are $x$-edges and oriented edges are  $y$-edges.} \label{sphere and torus} 
\end{figure}

\begin{lemma}
The 2-complex $J$ defined above is homeomorphic to a compact orientable surface without boundary.  \\[-18pt] 
\end{lemma}

\begin{proof}

In the construction of $J$ we identified oppositely oriented $x$-edges in pairs.  
For this proof, however, it is convenient to revert to the pairs of oriented $x$-edges,  
and insert a digon (which we call an $x$\emph{-digon}) between each pair.  
We will show that the resulting complex gives an  orientable surface without boundary.  
It will follow that the same is true of a januarial, because we have an embedding  in 
the same surface when the digons (any two of which have no $x$-edge in common) 
are successively collapsed to single edges.    

Now in this complex, each vertex has valence four: it has both an incoming and an 
outgoing $x$-edge (coming from an edge-loop in the event that $x$ fixes the vertex),  
and both an incoming and an outgoing  $y$-edge (which, similarly, may come from a loop). 
Each $x$-edge is incident with exactly one $x$-face (that is, an $x$-monogon or an $x$-digon), 
and one $xy$-face.  Each $y$-edge is incident with exactly one $y$-face and one $xy$-face.   
It follows that the complex gives a surface without boundary.  Moreover, the surface is 
orientable because the directions of the edges give consistent orientations around the 
perimeters of all the faces.   
Finally, since $S$ is finite, the surface is compact.
\end{proof}
${}$\\[-40pt]

\section{The topological complexity of januarials} \label{topological structure}

Higman gave a notion of topological complexity which we call \emph{simple type} below.    
It concerns how a januarial $J$ is assembled from the subspaces $S_1$ and $S_2$ that are the closures of its two $xy$-faces.    He recognised that some januarials are beyond the scope of this notion; indeed, he made some ad hoc calculations for the examples in Figures~\ref{43_associate} and  \ref{PSL31_associate}  which show as much.   Accordingly, below, we define a more general notion which we call \emph{type}, which applies to all januarials,  and we explain how to calculate it in general.

Topological features of $J$, $S_1$ and $S_2$ come into clearer focus when we collapse each $x$-monogon and each $y$-face in $J$ to a point.   Any two $y$-faces in a januarial are disjoint. The same is true of any two $x$-monogons.  And an $x$-monogon can only meet a $y$-face at a single vertex.  So these collapses do not change the homeomorphism types of $J$, $S_1$ or $S_2$.  

Let $\Gamma$ be the 1-skeleton of $J$ --- that is, the coset graph.  
 Let  $\bar{J}$, $\bar{S_1}$, $\bar{S_2}$ and $\bar{\Gamma}$ be the images of $J$, $S_1$, $S_2$ and $\Gamma$ under these collapses.  We call $\bar{\Gamma}$ a \emph{companion graph}.  Then  $\bar{J} = \bar{S_1} \cup \bar{S_2}$ is a closed surface obtained by some identification of  $\bar{S_1}$ and  $\bar{S_2}$ along their boundaries.  Taking another perspective, $\bar{J}$ is the result of adding two faces to $\bar{\Gamma}$, via attaching maps $\rho_1$ and $\rho_2$ induced by the maps that attach  the $xy$-faces to $\Gamma$.     

Examples of such $\Gamma$ and $\bar{\Gamma}$  appear in Figures~\ref{13 associate}, \ref{17 associate},  \ref{PSL73 associate figure},  \ref{43_associate},   \ref{PSL31_associate}, \ref{A_16}, and \ref{11}.  Each one is drawn in such a way that the cyclic order in which edges emanate from vertices agrees with that in which $x$-edges meet $y$-faces in $\Gamma$.  So, as the $y$-edges in $\Gamma$ are oriented anti-clockwise around the $y$-faces, one can read off $\rho_i$ by following successive edges in $\bar{\Gamma}$ in such a way that on arriving at a vertex, one exits along the right-most of all the remaining edges (except if the vertex has valence one, in which case one exits along the edge by which one arrived).   

As $\langle xy \rangle$ yields exactly two orbits when acting on $S$, together $\rho_1$ and $\rho_2$ traverse each edge in $\bar{\Gamma}$ twice, once in each direction.  The edges comprising the subgraph $\mathcal{G} := \bar{S_1} \cap \bar{S_2}$, shown in blue in the figures, are traversed by  $\rho_1$ in one direction and  $\rho_2$ in the other.  Those traversed by $\rho_1$ (resp.\ $\rho_2$) in both directions are shown in red (resp.\ green).   

The collapses carrying $J$ to $\bar{J}$ leave only the two $xy$-faces, those $x$-edges which are not loops, and one vertex  for each $y$-face in $J$.  These collapses do not alter the Euler characteristic.
Since $J$ is a closed orientable surface, we find that the genus of $J$ is readily calculated as follows.

\begin{lemma} \label{genus}
Twice the genus of a januarial equals the number of $x$-edges which are not loops minus the number of $y$-faces.   
\end{lemma}     

For example, this is $6-6=0$ in the left-hand example of Figure~\ref{sphere and torus}  and is $8-6=2$ in the right-hand example, giving genera $0$ and $1$, respectively.   

Now we turn to genera associated to $S_1$ and $S_2$, or their images $\bar{S_1}$ and  $\bar{S_2}$.  
Defining these requires care, since $\bar{S_1}$ and $\bar{S_2}$ may fail to be sub-surfaces of $\bar{J}$ (and likewise $S_1$ and $S_2$ fail to be sub-surfaces of $J$): they are closed surfaces from which the interiors of some collection of discs have been removed, but  the boundaries of those discs need not be disjoint.  (Figures~\ref{43_associate} and  \ref{PSL31_associate} provide such examples.)    But if we take a small closed neighbourhood $R_i$  of $\bar{S_i}$ in $\bar{J}$, we get a genuine sub-surface which serves as a suitable proxy:  

\begin{lemma} \label{surface}
$R_1$ and $R_2$ are orientable surfaces, and they retract to  $\bar{S_1}$ and $\bar{S_2}$, respectively.    \\[-18pt] 
\end{lemma}

\begin{proof}
A small closed neighbourhood of $\mathcal{G}$ (or indeed  of any subgraph of the 1-skeleton of a finite cellulation of a closed surface) is  a sub-surface with boundary and retracts to $\mathcal{G}$.  
Similarly  $R_i$, which is the union of $\bar{S_i}$ with a small closed neighbourhood of  $\mathcal{G} $, is orientable and retracts to  $\bar{S_i}$.   It is orientable because $\bar{J}$ is orientable. 
\end{proof}

We define the  \emph{type} of $J$ to be the pair $((h_1,g_1), (h_2,g_2))$, 
where $g_i$ and $h_i$ are the genus of $R_i$  and the number of connected 
components of the boundary of $R_i$ respectively,  for $i=1,2$.  
We will not distinguish 
between types $((h_1,g_1), (h_2,g_2))$ and $((h_2,g_2), (h_1,g_1))$.

The most straightforward way in which $R_1$ and $R_2$ can be assembled to make $\bar{J}$ occurs when 
 $R_1 \cap R_2$ is a disjoint union of $h$ annuli, where $h = h_1 = h_2$, or in other words, when $\bar{J}$ is homeomorphic to a join of $R_1$ and $R_2$ in which the boundary components are paired off and identified.   In this case, we say that the januarial $J$ is of \emph{simple type} $(h,g_1,g_2)$.  We do not distinguish between the simple types $(h,g_1,g_2)$ and $(h,g_2,g_1)$.   
  
Maps in which the graph is  embedded in a suitably non-pathological manner 
(for instance as a subgraph of the $1$-skeleton of a finite cellulation of the surface)  
have the property that a small neighbourhood is a disjoint union of annuli if and only if the graph is a collection of disjoint simple circuits.   So, as $R_1 \cap R_2$ is a small neighbourhood of $\mathcal{G}$, one can recognise simple type from the graph $\bar{\Gamma}\,$: 
 
 \begin{lemma} \label{simple type}
 $J$ is of simple type if and only if $\mathcal{G}$ is a collection of disjoint simple circuits.  
 In that case, if $J$ has simple type $(h,g_1,g_2)$ then $h$ is the number of circuits.  
 \end{lemma}

The genus of a januarial $J$ (equivalently, of  $\bar{J}$) of simple type is present in the data $(h,g_1,g_2)$.  When the handles (that is, the $h$ annuli from $R_1 \cap R_2$) that connect ${R_1}$ and ${R_2}$  are severed one-by-one, the genus falls by $1$ each time, until we only have one handle connecting  ${R_1}$ and ${R_2}$, and hence a surface of genus $g_1+ g_2$.  Since $J$ has $h$ handles to begin with, this gives the following:

\begin{lemma} \label{genus from type} The genus of a januarial $J$ of simple type  is 
 $g_1 +g_2 +h -1$.
\end{lemma}

Figures~\ref{13 associate}, \ref{17 associate}, \ref{PSL73 associate figure}  and \ref{A_16} show examples of coset graphs which give januarials of simple type, and Figures~\ref{43_associate},  \ref{PSL31_associate} and \ref{11} show examples  which give januarials that are not of simple type.  In each case, the caption indicates the genus of the januarial and the details of the type.  The genus of the januarial can be established in each case  via  an Euler characteristic calculation (as per Lemma~\ref{genus from type} for those of simple type).  

For the examples of simple type, $h$ is   immediately evident from the companion graph  $\bar{\Gamma}$ on account of Lemma~\ref{simple type}.  For those not of simple type, our 
next lemma gives a means of identifying $h_1$ and $h_2$ from $\bar{\Gamma}$.   Examples of \emph{partitions of $\mathcal{G}$ into circuits} in the sense of this lemma can be seen in Figures~\ref{43_associate},  \ref{PSL31_associate} and \ref{11}. 

 \begin{lemma} \label{finding h_i}
Let $\mathcal{P}$ be the set of all paths that traverse successive edges in $\mathcal{G}$
in the directions they are traversed by $\rho_1$ (resp.\ $\rho_2$), in such a way
that whenever such a path reaches a vertex, it continues along the right-most of the
other edges in $\mathcal{G}$ incident with that vertex.
(The next edge is necessarily  traversed by $\rho_1$ (resp.\ $\rho_2$) in that direction.)
All such paths close up into circuits, and $\mathcal{P}$  partitions $\mathcal{G}$,
in the sense that the  union of the circuits is $\mathcal{G}$ and no two share an edge.
The cardinality of  $\mathcal{P}$ is  $h_2$ (resp.\ $h_1$).   \\[-18pt]
\end{lemma}

\begin{proof}
We will prove the result for $\rho_1$.  The same argument holds for $\rho_2$ with the  subscripts $1$ and $2$ interchanged.  

By construction, the portion of the circuit $\rho_1$ that falls in $\mathcal{G}$ runs close alongside the boundaries of the $h_2$ holes in $R_2$.   Consider the situation  where $\rho_1$ is traversing an edge $e$ in  $\mathcal{G}$, and let $B$ denote the boundary of  the  hole that runs alongside --- see Figure~\ref{progress}.  At the terminal vertex $v$ of $e$, because of our convention for drawing companion graphs,   
$\rho_1$ will continue along the right-most of the other incident edges in $\bar{\Gamma}$.  If that edge $e'$ is in $\mathcal{G}$, it also runs alongside $B$.  (This happens at $u$ in the figure.) Suppose, on the other hand, that $e'$ is not in $\mathcal{G}$.  Then $\rho_1$ does not run alongside $B$, but rather heads into the interior of $R_2$.  (This happens at $v$ in the figure.)  At some later time, $\rho_1$ must return along $e'$ in the opposite direction (perhaps visiting another  portion of $B$ in the interim) since the edges $\rho_1$ traverses exactly once are precisely those   in $\mathcal{G}$.  Hence $\rho_1$ arrives back at $v$ and then continues along the (new) right-most edge --- which will either be alongside $B$, or take it back into the interior of $R_2$, again to   return eventually along that same edge.  Repeating this, we eventually find the next edge in $\mathcal{G}$ incident with $v$ that continues alongside $B$.   
It follows that however many detours into the interior of $R_2$ are required,  it is the \emph{right-most of the edges in $\mathcal{G}$} incident with $v$ (aside from $e$) that continues alongside $B$.  So the circuits traversed as explained in the statement of the lemma are those that run alongside the boundaries of the holes in $R_2$.  The remaining claims easily follow from this.   
\end{proof}

\begin{figure}[ht]
\psfrag{v'}{{$v$}}
\psfrag{v}{{$u$}}
\psfrag{B}{$B$}
\psfrag{S}{$\bar{S_1}$}
\centerline{\epsfig{file=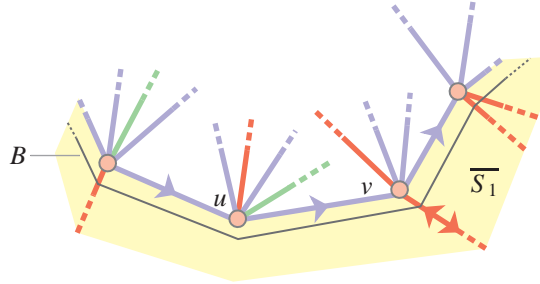}} 
\caption{Tracking the boundary of one of the holes in $R_2$.} \label{progress}
\end{figure}

Given $h_i$ (for $i = 1$ or $2$), one can determine $g_i$ from $\bar{\Gamma}$ via the following observation: 

\begin{lemma} \label{finding g_i}
The genus $g_i$ of $R_i$  satisfies $2-2g_i = V_i  - E_i + h_i + 1$ where $V_i$ and $E_i$ denote the number of vertices and edges, respectively, in the  subgraph of $\bar{\Gamma}$ visited by the attaching map of the face of $\bar{S_i}$.   \\[-18pt] 
\end{lemma}

\begin{proof}
  By Lemma~\ref{surface},  filling the $h_i$ holes in $R_i$ with discs gives a closed orientable surface of genus $g_i$ which is homotopic to  $\bar{S_i}$  with $h_i$ discs attached along circuits in its $1$-skeleton.   
Hence the Euler characteristic $2-2g_i$ of $R_i$ is the same as that of $\bar{S_i}$ 
with the $h_i$ discs attached, namely  $V_i  - E_i + h_i + 1$.
\end{proof}

\begin{Questions}\label{open qus} Higman asked the following questions concerning $k$-januarials of simple type.  For a given $k$, what are the possible values for and interrelationships between  $g_1$,  $g_2$ and $h$?  Are there arbitrarily large values of $k$ for which there exist  examples with $h=1$?  
How large can $h$ be, for given $k$? 
Similar questions can be asked also about januarials that are not of simple type. 
 \end{Questions}
${}$\\[-40pt]

\section{\texorpdfstring{Januarials from $\PSL(2,q)$}{Januarials from PSL(2,q)}} \label{finding}

\subsection{\texorpdfstring{$\PGL(2,q)$, $\PSL(2,q)$ and the classical modular group}{PGL(2,q), PSL(2,q) and the classical modular group}}
\label{findingfirst}

The \emph{projective linear group} $\PGL(n,\mathbb{F})$ over a field $\mathbb{F}$ 
is the quotient $\GL(n,\mathbb{F})/Z$ of the group of invertible $n \times n$ matrices 
by its centre $Z = \set{\, a I_n \mid a \in \mathbb{F} \ssm \set{0}}$. 
Its subgroup, the \emph{projective special linear group} $\PSL(n,\mathbb{F})$, is the 
quotient of $\SL(n,\mathbb{F})$, the group of all $n \times n$ matrices  over  $\mathbb{F}$ 
of determinant one, by its subgroup of all scalar matrices of determinant one.

There is a natural isomorphism between $\PGL(2,\mathbb{F})$  and a group of M\"{o}bius 
transformations, under which 
 the matrix  
$\, \begin{pmatrix} 
 a & b \\ 
 c & d
 \end{pmatrix} \, $ \ 
corresponds to the transformation $\, {\ds z \mapsto \frac{az +c}{bz+d}}\, ,$  
when multiplication of transformations is read from left to right. 
This gives actions of $\PGL(2,\mathbb{F})$ and $\PSL(2,\mathbb{F})$ on the projective 
line  $\PL(\mathbb{F}) = \mathbb{F} \cup \set{\infty}$.  
Also if $\mathbb{F}$ is finite, of order $q$, then $\PGL(2,\mathbb{F})$ and $\PSL(2,\mathbb{F})$ 
are denoted by $\PGL(2,q)$ and $\PSL(2,q)$.

A search for 3-januarials may begin with the \emph{classical modular group} 
$$
\PSL(2,\Z) \ = \  \langle \, x, y \, | \,    x^2 =y^3=1 \rangle 
$$ 
which acts on $\Q \cup \set{\infty}$ by M\"obius transformations 
with $x: z \mapsto -1/z$ and $y: z \mapsto (z-1)/z$.  
Notice that $x y : z \mapsto z +1$. 
A portion of the resulting coset diagram is shown in Figure~\ref{PSLZ coset diagram}.

\begin{figure}[ht]
\psfrag{x}{}
\psfrag{t}{$t: z \mapsto 1/z$} 
\psfrag{0}{\tiny{$0$}}
\psfrag{1}{\tiny{$1$}}
\psfrag{2}{\tiny{$2$}}
\psfrag{3}{\tiny{$3$}}
\psfrag{4}{\tiny{$4$}}
\psfrag{5}{\tiny{$5$}}
\psfrag{-1}{\tiny{$-1$}}
\psfrag{-2}{\tiny{$-2$}}
\psfrag{-3}{\tiny{$-3$}}
\psfrag{-4}{\tiny{$-4$}}
\psfrag{1/5}{\tiny{$1/5$}}
\psfrag{5/4}{\tiny{$5/4$}}
\psfrag{1/4}{\tiny{$1/4$}}
\psfrag{-34}{\tiny{$-3/4$}}
\psfrag{1/3}{\tiny{$1/3$}}
\psfrag{4/3}{\tiny{$4/3$}}
\psfrag{7/3}{\tiny{$7/3$}}
\psfrag{-23}{\tiny{$-2/3$}}
\psfrag{-53}{\tiny{$-5/3$}}
\psfrag{3/8}{\tiny{$3/8$}}
\psfrag{-25}{\tiny{$-2/5$}}
\psfrag{3/5}{\tiny{$3/5$}}
\psfrag{8/5}{\tiny{$8/5$}}
\psfrag{5/7}{\tiny{$5/7$}}
\psfrag{-5/2}{\tiny{$-5/2$}}
\psfrag{-32}{\tiny{$-3/2$}}
\psfrag{-1/2}{\tiny{$-1/2$}}
\psfrag{5/2}{\tiny{$5/2$}}
\psfrag{3/2}{\tiny{$3/2$}}
\psfrag{1/2}{\tiny{$1/2$}}
\psfrag{7/2}{\tiny{$7/2$}}
\psfrag{-3/5}{\tiny{$-3/5$}}
\psfrag{2/5}{\tiny{$2/5$}}
\psfrag{7/5}{\tiny{$7/5$}}
\psfrag{5/8}{\tiny{$5/8$}}
\psfrag{8/3}{\tiny{$8/3$}}
\psfrag{5/3}{\tiny{$5/3$}}
\psfrag{2/3}{\tiny{$2/3$}}
\psfrag{-13}{\tiny{$-1/3$}}
\psfrag{-4/3}{\tiny{$-4/3$}}
\psfrag{2/7}{\tiny{$2/7$}}
\psfrag{3/7}{\tiny{$3/7$}}
\psfrag{7/4}{\tiny{$7/4$}}
\psfrag{3/4}{\tiny{$3/4$}}
\psfrag{-14}{\tiny{$-1/4$}}
\psfrag{4/5}{\tiny{$4/5$}}
\psfrag{4/7}{\tiny{$4/7$}}
\psfrag{i}{\tiny{$\infty$}}
\centerline{\epsfig{file=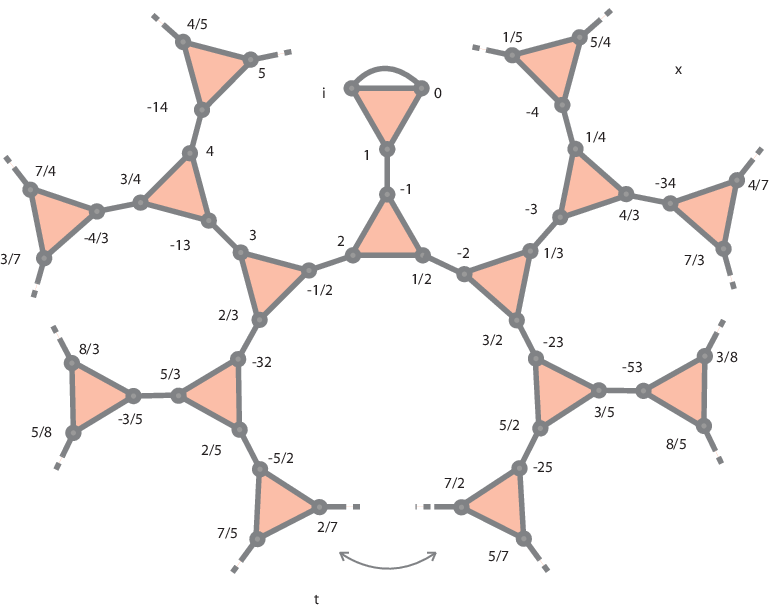}} 
\caption{The coset graph arising from the action of  the modular group 
$\PSL(2,\Z)  =  \langle \, x,y \, | \,  x^2 = y^3 = 1 \, \rangle$ on $\mathbb{Q} 
\cup \set{\infty}$ via $x : z \mapsto -1/z$ and $y : z \mapsto (z-1)/z$.} 
\label{PSLZ coset diagram}
\end{figure}

Suppose $p$ is a prime. 
Then the group $\PSL(2,p)$ is a homomorphic image of  
$$
\Delta(2,3,p)  \  = \  \langle \,  x,y \,  \mid \,  x^2 = y^3 = (xy)^p =1 \, \rangle, 
$$  
whereby $\Delta(2,3,p)$ acts on $\F_p \cup \set{\infty}$ 
via $x : z \mapsto -1/z$ and $y : z \mapsto (z-1)/z$, with product $x y : z \mapsto z +1$.  
The resulting coset diagrams $\Gamma$ are quotients of the diagram in Figure~\ref{PSLZ coset diagram};   
for example, the diagrams for  $\PSL(2,{13})$ and $\PSL(2,{17})$ are shown in 
Figure~\ref{13 and 17 coset diagram}.   
But, these coset diagrams do not immediately yield  januarials,  
since the orbits of  $\langle xy \rangle$ have lengths $1$ and $p$ rather than both $(p+1)/2$. 

\begin{figure}[ht]
\psfrag{x}{}
\psfrag{t}{$t: z \mapsto 1/z$} 
\psfrag{i}{\tiny{$\infty$}}
\psfrag{0}{\tiny{$0$}}
\psfrag{1}{\tiny{$1$}}
\psfrag{2}{\tiny{$2$}}
\psfrag{3}{\tiny{$3$}}
\psfrag{4}{\tiny{$4$}}
\psfrag{5}{\tiny{$5$}}
\psfrag{6}{\tiny{$6$}}
\psfrag{7}{\tiny{$7$}}
\psfrag{8}{\tiny{$8$}}
\psfrag{9}{\tiny{$9$}}
\psfrag{10}{\tiny{$10$}}
\psfrag{11}{\tiny{$11$}}
\psfrag{12}{\tiny{$12$}}
\centerline{\epsfig{file=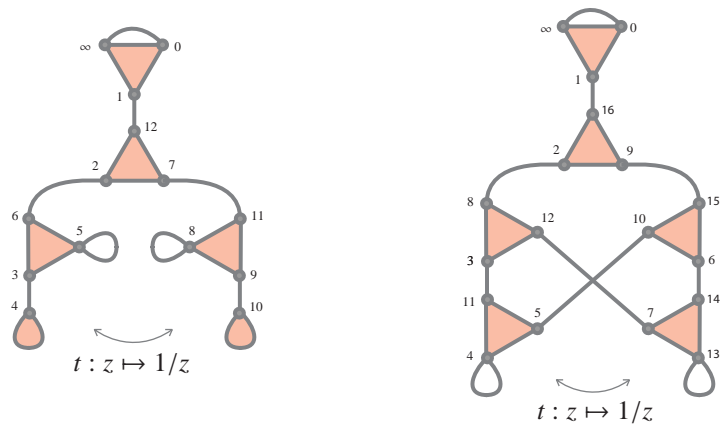}} 
\caption{Coset graphs arising from actions of $\Delta(2,3,13)$ 
on $\F_{13} \cup \set{\infty}$ and $\Delta(2,3,17)$ 
on $\F_{17} \cup \set{\infty}$, 
both via  $x : z \mapsto -1/z$ and $y : z \mapsto (z-1)/z$.} 
\label{13 and 17 coset diagram}
\end{figure}

\subsection{Associates} \label{associates}

Here is a potential remedy for the failure of the   coset diagrams constructed above from $\PSL(2,p)$ to produce januarials.  It applies in the general setting of a finite 
group $G$ acting on a set $S$ and containing elements $x$ and $y$ satisfying 
$x^2  =  y^k = (x y)^l =1$ for some  $l \in \Z$.  
Let $\Gamma$ be the resulting coset graph for the action of $\Delta(2,k,l)$ on $S$, 
via $G$, with respect to the generating set $\set{x,y}$.
 
Suppose $G$ has an element $t$ of order $2$ with the property that 
$t^{-1}xt=x^{-1}$ ($= x$) and $t^{-1}yt=y^{-1}$. 
Conjugation by such an element $t$ reverses every cycle of $y$ and preserves  
every cycle of the involution $x$, and hence $t$ induces a reflection of the 
coset graph $\Gamma$. 

Note that $(xt)^2  =  x(txt) = xx = 1$, which allows us to consider the pair $(xt,y)$ in place of $(x,y)$. 
If $l'$ is the order of $xt\, y$, then we have an action of $\Delta(2,3,l')$ on $S$ via 
$\langle xt,y\rangle$.  
The  resulting coset graph  $\Gamma'$  is called an \emph{associate} of $\Gamma$. 
This graph also admits a reflection via the same involution $t$, since 
$t^{-1}(xt)t = x^{-1}t = xt = (xt)^{-1}$ (and $t^{-1}yt=y^{-1}$). 
For more details about the correspondence $(x,y,t) \mapsto (xt,y,t)$, see \cite{JTh}. 
The associate graph $\Gamma'$ gives a new candidate for a januarial.  

\begin{Examples}  
When $G$ is $\PSL(2,p)$ for some prime $p \equiv 1$ mod $4$, 
and $x : z \mapsto -1/z$ and $y : z \mapsto (z-1)/z$, we can take $t$ to be the 
transformation $z \mapsto 1/z$, which has order $2$ and satisfies 
$t^{-1}xt=x^{-1}$ and $t^{-1}yt=y^{-1}$, as required.    
[The condition $p \equiv 1$ mod $4$ ensures that $-1$ is a square mod $p$, 
so that the transformation $t$ (and hence also $xt$) lies in $\PSL(2,p)$.]   
In this case, $xt$ is  the transformation $z \mapsto -z$. 
Hence, in particular, the associates of the coset graphs in the cases $p = 13$ and $17$ 
from Figure~\ref{13 and 17 coset diagram} are precisely those in Figures~\ref{13 associate} 
and \ref{17 associate}.  
The transformation $xty : z \mapsto (z+1)/z$ has two cycles of equal length, 
and so in both cases the associates are januarials --- specifically those depicted 
in Figure~\ref{sphere and torus}. 
\end{Examples}

Like $\Gamma$, the associate can fail to yield a januarial if the sizes of the $\langle xy \rangle$-orbits are not the requisite $\abs{S}/2$,  but  it succeeds in many cases.   In  Section~\ref{when the associate succeeds} we will explore when it can be successfully applied to the examples from $\PSL(2,p)$.  But first we need the following study of conjugacy classes in $\PGL(2,q)$.

\subsection{\texorpdfstring{Classifying conjugacy classes in $\PGL(2,q)$}{Classifying conjugacy classes in PGL(2,q)}}
 
Some of the details of the analysis in this section are similar to that carried 
out by Macbeath in \cite{Mac}.  
 
Let $q$ be any odd prime-power greater than $3$, say $q = p^s$.
For $M = \begin{pmatrix}
  a & b \\
  c & d
\end{pmatrix}  \in \GL(2,q)$, 
we may define   
$$ 
\theta (M) \ = \ 
\frac{(\tr\, M)^2}{\det\, M\,} \  = \  \frac{(a+d)^2}{ad-bc}.
$$ 

The characteristic polynomial of  $M$   
is $\,\det(xI_2 - M) = x^2 - {\rm tr}(M)x + \det(M),\,$ 
and ${\rm tr}(M)$ and $\det(M)$, and therefore also $\theta(M)$,  
are  invariant under conjugacy  within $\GL(2,q)$.    
Since $\theta(M)$ is invariant under scalar multiplication of $M$, it follows that 
this gives us a well-defined function $\theta : \PGL(2,q) \to \mathbb{F}_q$ 
that is constant on conjugacy classes of $\PGL(2,q)$. 

In fact, the function $\theta$ parametrises the conjugacy classes of $\PGL(2,q)$, as follows.  

\begin{prop} \label{classification}
Let  $g$ and $h$ be elements of $\PGL(2,q)$. 
If $\,\theta(g) = \theta(h) \not\in \{0,4\},\,$ then $g$ and $h$ are   conjugate in $\PGL(2,q)$.
In the exceptional cases, there are precisely two conjugacy classes on which $\theta = 0$, 
namely the class of involutions in $\PSL(2,q)$ and the class of involutions 
in $\PGL(2,q) \ssm \PSL(2,q),$
and two classes on which $\theta = 4$, namely the 
class containing the identity element and the class of the transformation $z \mapsto z+1$. \\[-18pt] 
\end{prop}

This proposition can be proved using rational canonical forms, but also we can give a direct proof for the 
generic case.  

\begin{proof}[Proof for the case where $\theta(g) \notin \set{0,4}$.]
Suppose the transformation $g \in \PGL(2, q)$ is induced by the matrix $M \in \GL(2,q)$, 
with  $\tr(M) \neq 0$.  Then we can choose a vector $u \in \F_q^{\,2}$ 
such that $u$ and $u M$ are linearly independent over $\F_q$.  
The matrix for $g$ with respect to the basis $\set{ u , u M }$ is the of the form   
$\begin{pmatrix}
  0 & 1 \\
  \ast & \ast
\end{pmatrix},$
but since the trace is non-zero and a conjugacy invariant, we can change the basis 
if necessary, so that the matrix for $g$ has entry $1$ in the lower-right corner.  
The matrix for $g$ then becomes 
$$
M'  \ =  \ \begin{pmatrix}
  0 & 1 \\
  - \Delta & 1
\end{pmatrix}, \\[+4pt] 
$$
where $\Delta$ is the determinant.  
But then $\theta(g) = \theta(M') = 1/\Delta$, so $\Delta = 1/\theta(g)$, 
and it follows that $\theta(g)$ determines the matrix. 
Since matrices representing the same linear transformation  with respect to different 
bases are conjugate within $\GL(2,q)$, we find that $\theta(g)$ determines the 
conjugacy class of $g$. 
\end{proof}

The utility of the parameter $\theta$ is enhanced by the following lemma, 
which gives a number of cases in which the order of an element $y \in \PGL(2,q)$ determines $\theta(y)$.

Elements with trace $0$ give involutions in $\PGL(2,q)$, and conversely, 
while elements with trace $-1$ and determinant $1$ give elements 
of order $3$ in $\PGL(2,q)$, and parabolic elements   (which are the conjugates 
of the matrix $\begin{pmatrix}
  1 & 0 \\
  1 & 1
\end{pmatrix}$, or equivalently, the elements 
with trace $2$ and determinant $1$), give elements of order $p$ in $\PGL(2,q)$. 
We also note that every element of order $(q+1)/2$ in $\PGL(2,p)$ is the square of 
an element of order $q+1$ in $\PGL(2,q)$, and hence lies in $\PSL(2,q)$   
and is the product of two cycles of length $(q+1)/2$ in the natural action of $\PGL(2,q)$ on $\mathbb{F}_q \cup \set{\infty}$.

\begin{lemma} \label{theta from orders}
If $y$ is an element of order $1$, $2$, $3$, $4$ or $6$ in $\PGL(2,q)$, 
then $\theta(y) =  4$, $0$, $1$, $2$ or $3$, respectively. 
Also if $y$ has order $p$ (the prime divisor of $q$) then $\theta(y) = 4$. \\[-18pt]
\end{lemma} 

\begin{proof} 
Suppose $y$ is induced by the element $M = \begin{pmatrix}
  a & b \\
  c & d
\end{pmatrix}  \in \GL(2,q)$. 
Then the first three cases are easy consequences of the respective observations 
that in those cases, $M$ is scalar, or $M$ has trace $0$, or $M$ has minimum polynomial 
$x^2+x+1$.

For the next two cases, we note that 
$\, M^2 = \begin{pmatrix}
  a^2+bc & b(a+d) \\
  c(a+d) & d^2+bc
\end{pmatrix} \, $
and therefore 
$$\tr(M^2) \ = \ a^2+d^2+2bc \ = \ (a+d)^2 - 2(ad-bc) 
\ = \ (\tr\,M)^2 - 2\,\det\,M.$$

If $g$ has order $4$, then $M^2$ has order $2$, 
and so $0 = \tr(M^2) = (\tr\,M)^2 - 2\,\det\,M$, 
which gives $\theta(y) = \theta(M) = (\tr\, M)^2/\det\, M = 2$. 
Similarly, if $y$ has order $6$, then since $M^2$ induces an element of order 3 
in $\PGL(2,q)$ we know that 
$$
((\tr\,M)^2 - 2\,\det\,M)^2 \ = \ (\tr(M^2))^2 \ = \ \det(M^2) \ = \ (\det\,M)^2, 
$$ 
and therefore $(\theta(M)-2)^2 = 1$. But since $y$ does not have order $3$, 
we know that $\theta(M) \ne 1$, and so $\theta(M)-2 = 1$, which gives $\theta(y) = \theta(M) = 3$.

Finally, if $y$ has order $p$, then $y$ is parabolic and therefore induced by 
some conjugate of the matrix $\begin{pmatrix}
  1 & 0 \\
  1 & 1
\end{pmatrix}$, 
which implies that $\theta(y) = (1+1)^2/1 = 4$. 
\end{proof}

\subsection{\texorpdfstring{How many $3$-januarials arise from $\PSL(2,p)$?}{How many 3-januarials arise from PSL(2,p)?}} 

We can now proceed further, to consider $3$-januarials.  
Let $\phi$ be Euler's totient function --- that is, let $\phi(n)$ be the number of 
integers in $\set{1, \ldots, n}$ that are coprime to $n$.  
 
\begin{lemma}
The number of conjugacy classes of elements in $\PGL(2,q)$ of
order $(q+1)/2$ is $\frac{1}{2}\phi((q+1)/2)$. Moreover, if $z$ is any element of 
order  $(q+1)/2$ in $\PGL(2,q)$, then every one of these conjugacy classes 
intersects the subgroup generated by $z$ in $\set{z^i,z^{-i}}$ for exactly 
one $i$ coprime to $(q+1)/2$. \\[-18pt] 
\end{lemma}

\begin{proof}
This follows easily from the observation that every element of order $(q+1)/2$ 
in $\SL(2,q)$ is conjugate in $\GL(2,q^2)$ to one of the form 
$\, \begin{pmatrix} 
 \lambda^i & 0 \\ 
 0 & \lambda^{-i}
 \end{pmatrix}, \, $ 
where $\lambda$ is an element of order $(q+1)/2$ in the field $\mathbb{F}_{q^2}$, 
and the traces $\lambda^i + \lambda^{-i}$ are distinct in $\mathbb{F}_q$. 
\end{proof}

\begin{lemma}\label{3jans-pgl2q} 
For any given conjugacy class $C$ of elements of $\PSL(2,q)$ of order $l \not\in \{1,2,p\}$, 
there exists a triple $(x,y,xy)$ of elements of $\PSL(2,q)$ such that 
$x$ has order $2$, and $y$ has order $3$, and $xy$ is in $C$. 
Moreover, this triple is unique up to conjugacy in $\PGL(2,q)$ 
whenever $l \ne 6$. \\[-18pt] 
\end{lemma}

\begin{proof}
Every element of order $3$ in $\PSL(2,q)$ is conjugate in $\PGL(2,q)$ to the element 
$y : z \mapsto (z-1)/z$, induced by the matrix 
$Y = \, \begin{pmatrix} 
-1 & 1 \\ 
-1 & 0
 \end{pmatrix}, \, $ 
while any element of order $2$ in $\PSL(2,q)$ is induced by a matrix  $X$ 
of the form
$X = \, \begin{pmatrix} 
a & b \\ 
c & -a
 \end{pmatrix}, \, $ 
with trace $0$ and determinant $-a^2-bc = 1$.  
Now observe that for any such choice of $X$ and $Y$, we have \ 
$XY \, = \, \begin{pmatrix} 
-a-b & a \\ 
\phantom{-}a-c & c
 \end{pmatrix}, \, $
%
which has trace $\,{\rm tr}(XY) = -a-b+c\,$ and determinant $\,\det(XY) = \det(X)\det(Y) = 1$. 

This can be turned around:  we can show that for any given non-zero trace $r$, 
there exist $a,b$ and $c$ in $\F_q$ such that $-1 = a^2+bc$ and $r = -a-b+c$, 
and hence there exists a matrix $X$ of trace zero such that $XY$ has trace $r$, 
giving a triple $(x,y,xy)$ of the required type. 
Note that we need $c-b = r+a$ and $bc = -(1+a^2)$, and therefore 
$(c+b)^2 = (c-b)^2+4bc = r^2 +2ar-3a^2-4$.
Now if $p \ne 3$, then we can multiply this by 3 
and it becomes $3(c+b)^2 = 4r^2 - 12 - (3a-r)^2$; 
and then since in $\F_q$ there are $(q+1)/2$ elements of the form $3u^2$ 
and $(q+1)/2$ elements of the form $s-v^2$ for any given $s \in \F_q$, 
and any two subsets of size $(q+1)/2$ in $\F_q$ have non-empty intersection, 
the latter equation can be solved in $\F_q$ for $c+b$ and $3a-r$, 
and hence for $c+b$ and $a$ (and $c-b = r-a$), and hence for $a$, $b$ and $c$ 
(since $q$ is odd). 
On the other hand, if $p = 3$, then the equation 
becomes $(c+b)^2 = r^2 +2ar-3a^2-4 = r^2 +2ar-1$, 
which is even easier to solve for $a,b,c$, provided that $r \ne 0$. 

The main assertion now follows easily. 
Uniqueness up to conjugacy in $\PGL(2,q)$ when $l \ne 6$ is left as an exercise. 
\end{proof}

Since the automorphism group of $\PSL(2,p)$ is $\PGL(2,p)$ for every odd prime $p$, 
we obtain the following: 
 
\begin{corollary}
For any  prime $p > 3$, the number of distinct  $3$-januarials 
constructible from $\PSL(2,p)$ in the way described in sub-sections~ {\em\ref{findingfirst}} 
and~{\em\ref{associates}} is $\frac{1}{2}\phi((p+1)/2)$.  
\end{corollary}

For example, when $p= 37$ the number of $3$-januarials is $\phi(19)/2 = 9$, 
and when $p = 53$ the number is $\phi(27)/2 = 9$ as well.  
A further attractive property is as follows: 

\begin{lemma}
For every triple $(x,y,xy)$ as in Lemma {\em\ref{3jans-pgl2q}} with $l \ne 6$,
there exists an involution $t$ in $\PGL(2,q)$ such that
$t^{-1} x t = x^{-1}$ and $t^{-1} y t =y ^{-1}$.
\end{lemma}

\begin{proof}
We may suppose that $x$ and $y$ are induced by the matrices $X$ and $Y$ as in the proof 
of Lemma {\em\ref{3jans-pgl2q}}.  
In that case, let 
$T =  \, \begin{pmatrix} 
-(b+c) & 2a+b \\ 
2a-c & b+c
 \end{pmatrix}, \, $
which has   the property that $XT = TX^{-1}$ while $YT = TY^{-1}$. 
The determinant of $T$ is 
$$-(b+c)^2 -(2a+b)(2a-c) \  = \ -4a^2 - b^2 - c^2 - 2ab +2ac - bc, $$ 
which equals $\,3\det(XY) - (\tr(XY))^2,\,$ since $\det(XY) = -a^2-bc$ while $\tr(XY) = -a-b+c$, 
and therefore $T$ is invertible if and only if $\theta(XY) = (\tr(XY))^2/\det(XY) \ne 3$, 
or equivalently, $XY$ does not have order $6$. 
Finally, note that 
if $T$ is invertible, then since its trace is zero, we have $t^2=1$.
\end{proof}


\subsection{Necessary conditions for associates to yield januarials} \label{when the associate succeeds}

For any positive integer $n$, define $\,\theta_{n}\,$to be the set of all values of $\theta(g)$  
for elements $g$ of order $n$ in the group $\PGL(2,q)$.  Consider the effect of the mapping $r \mapsto (r-1)^2$ on the 
elements of this set $\theta_n$, for some $n$, as follows. 

Suppose that $r \in \theta_n$, where $n$ is coprime to $q$, and let $g$ be an element of 
order $n$ in $\PGL(2,q)$ with $\theta(g) = r$.  By taking 
a conjugate of $g$ if necessary in $\PGL(q^2)$, we may assume that  
$g$ is the transformation $ z \mapsto \rho z$ induced by the matrix 
$$
M = \begin{pmatrix}
  \rho & 0 \\
  0 & 1 \end{pmatrix}, 
$$
where $\rho$ is a primitive $n$th root of $1$ in $\F_q$ or $\F_{q^2}$. 
In this case, $\tr(M) = \rho +1$ while $\det(M) = \rho$, and so 
$$ 
r  \ = \ \theta(g) \ = \  \theta(M) \ = \ \frac{(\rho +1)^2}{\rho} \  = \  \rho + \rho ^{-1} +2.
$$
It follows that $(r -2)^2  =  (\rho + \rho ^{-1})^2 = \rho^2 + \rho^{-2} +2$. 
But
$r$ is of order $n$ and so $\rho$ will be a primitive $n$th root of unity.  
In particular, if $n$ is odd then also $\rho^2$ is a primitive $n$th root of unity,   
in which case $(r-2)^2 = \rho^2 + \rho^{-2} +2 = \theta(M^2)$, which also 
belongs to $\theta_n$. 
Iterating the procedure then yields further elements of $\theta_n$, 
until we reach a stage where $\rho^{2^i} = \rho^{\pm 1}$, 
and then $\rho^{2^i} + \rho^{-2^i} +2 = \rho + \rho ^{-1} +2 = r$.

\ms

We now derive two necessary conditions on the values of $\theta(g)$ in $\theta_n$ 
in the special case where $n = (q+1)/2$.  

\begin{lemma} \label{necessities}
If $g$ is an element of order $(q+1)/2$ in $\PSL(2,q)$, 
then $\theta(g)$ is a square in $F_q$ while  $\theta(g)-4$ is not a square in $F_q$. \\[-18pt] 
\end{lemma}

\begin{proof}
First, $g$ is conjugate in $\PGL(2,q^2)$ to the projective image of 
$M = \begin{pmatrix}
  \mu^2 & 0 \\
  0 & 1
\end{pmatrix}$
where $\mu$ is a primitive $(q+1)$th root of unity in $\mathbb{F}_{q^2}$, and   
this gives  
$$\theta(g) \ = \ \frac{(\tr\, M)^2}{\det\, M\,}  \ = \ \frac{(\mu^2+1)^2}{\mu^2} \ = \ \mu^2 + 2 + \mu^{-2}.$$ 

In particular, $\theta(g) = \mu^2 + 2 + \mu^{-2} = (\mu+\mu^{-1})^2$, which is a square in $\F_q$ 
(since $\mu+\mu^{-1} \in \F_q$). 
On the other hand, $\theta(g)-4 = \mu^2  - 2 + \mu^{-2} = (\mu- \mu^{-1})^2$, 
which is not a square in $\F_q$, since $\mu-\mu^{-1} = 2\mu - (\mu+\mu^{-1}) \not\in \F_q$. 
\end{proof}
 
\begin{cor} \label{conditions on p}
Suppose  $g \in \PSL(2,p)$ has order $(p+1)/2$.    \\[-18pt] 
\begin{enumerate}
\item[{\rm (a)}] If $\theta(g)=-1$, \ then $p \equiv 13$ or $17$  mod $20$.
\item[{\rm (b)}] If $\theta(g)=-2$, \ then $p \equiv 17$ or $19$  mod $24$.
\item[{\rm (c)}] If $\theta(g)=-3$, \ then $p \equiv 13$, $19$ or $31$  mod $42$. \\[-18pt] 
\end{enumerate}
\end{cor}

\begin{proof} 
In case (a), by Lemma~\ref{necessities} we require that $-1 = \theta(g)$ is  a square mod $p$ 
while $-5 = \theta(g) - 4$ is not, and hence also $5$ is not.  
Thus $p \equiv1$ mod $4$, and by quadratic reciprocity, also $p \equiv 2$ or $3$ mod $5$, 
giving $p \equiv 13$ or $17$  mod $20$. 
Similarly, in case (b) we require that $-2$ is a square mod $p$ while $3$ is not.  
It follows that $p \equiv 1$ or $3$ mod $8$, while also $p \not\equiv \pm 1$ mod $12$, 
and therefore $p \equiv 17$ or  $19$ mod $24$ (since we are assuming $p > 3$).   
Finally, in case (c) we require that $-3$ is a square mod $p$ while $-7$ is not, 
and hence that $p$ is a square mod $3$ and a non-square modulo $7$, giving 
$p \equiv 10$, $13$ or $19$ mod $21$, and therefore $p \equiv 13$, $19$ or $31$  mod $42$.
\end{proof}

Next, we give what Higman described as the `Pythagorean Lemma'.   
One motivation for this choice of name is that for an element $X$ of ${\rm SO}(3)$, we could 
define $\theta(X)$ to be $4\cos^2(\phi/2)$, where $\phi$ is the angle of rotation of $X$.   
Now let $a$, $b$ and $c$ be half-turns about the three co-ordinate axes, 
and let $d$ be a half-turn about any unit vector $(\cos \alpha, \cos \beta, \cos \gamma)$.   
Then the angle of rotation of $ad$ is twice the angle between the 
axes of $a$ and $d$, namely $2\alpha$, and similarly the angles of rotation of $bd$ and $cd$ 
are $2\beta$ and $2\gamma$. 
Thus $\theta(ad) + \theta(bd) + \theta(cd) = 4\cos^2 \alpha + 4\cos^2 \beta + 4\cos^2 \gamma =  4$.

\begin{lemma}[Pythagorean Lemma]  Suppose $a$, $b$ and $c$ are the
non-identity elements of a subgroup of $\PSL(2,q)$ isomorphic to the Klein $4$-group.  
If $d$ is any element of order $2$ in $\PSL(2,q)$, then $$ \theta (ad) + \theta (bd) + \theta (cd) \  = \ 4.$$
\end{lemma}

\begin{proof}
One can take a quadratic extension $\mathbb{F}_{q^2}$ of the ground field $\mathbb{F}_q$, 
and then in $\PSL(2,\mathbb{F}_{q^2})$, all copies of the Klein $4$-group are
conjugate,   
since we are assuming that $q$ is odd.   
(See the classification of subgroups of $\PSL(2,p^f)$  
in \cite{Hup} for example.) 
Hence we may assume that our 
Klein $4$-group  in $\PGL(2,q)$ is generated by $$ a: z \mapsto
-z, \ \ \ \ \ \ b : z \mapsto -\frac{1}{z}, \ \ \ \hbox{ and } \ \ \ \ c: z
\mapsto \frac{1}{z}.$$  Now any involution $d$ has trace zero and hence is of the form \ 
${\displaystyle d: z \mapsto \frac{\alpha z + \beta}{ \gamma z - \alpha}.}$ 
\newline From this we find that 
$$
ad : z \mapsto
\frac{-\alpha z+ \beta}{-\gamma z - \alpha}, \ \ \ \ \ bd: z
\mapsto  \frac{-\alpha + \beta z}{-\gamma  - \alpha z}, 
\ \ \ \hbox{ and } \ \ \ \ 
cd: z \mapsto  \frac{\alpha + \beta z}{\gamma  - \alpha z}, 
$$ 
and therefore 
$$ 
\theta (ad) + \theta (bd) + \theta (cd) \  = \  \frac{4 \alpha^2}{\alpha^2 + \beta
\gamma} + \frac{(\beta - \gamma)^2}{-\beta \gamma - \alpha^2} +
\frac{(\beta + \gamma)^2}{\beta \gamma + \alpha^2} \ = \ 4,$$ as
required.
\end{proof}


Lemma~\ref{theta from orders}, the Pythagorean Lemma and Corollary~\ref{conditions on p} combine to give us necessary conditions for  associates formed as in Section~\ref{associates} to yield januarials.  
The scope of this result is limited to $k$ equal to  $3$, $4$ or $6$, 
since these are the only orders for which  Lemma~\ref{theta from orders}  applies.

\begin{cor}   \label{necessary conditions}
  Consider $\Delta(2,k,p)  \  = \  \langle \,  x,y \,  \mid \,  x^2 = y^k = (xy)^p =1 \, \rangle$ acting on $\F_{p} \cup \set{\infty}$ via $\PSL(2,p)$ in such a way that  $xy$ is the transformation $z \mapsto z+1$.  Let $t$ be an involution in $\PGL(2,p)$ such that $t^{-1}xt=x^{-1}$ and $t^{-1}yt=y^{-1}$, and suppose that the resulting associate coset diagram found by replacing  $x$ by $xt$ yields a $k$-januarial for $\PSL(2,p)$ 
or $\PGL(2,p)$, depending on whether or not $t$ lies in $\PSL(2,p)$.     
Then \\[-20pt] 
\begin{enumerate}
\item[{\rm (a)}]  if $k=3$, then  $p \equiv 13$ or $17$ mod $20\,;$ \\[-15pt]   
\item[{\rm (b)}]  if $k=4$, then    $p \equiv 17$ or $19$  mod $24\,;$\\[-15pt] 
\item[{\rm (c)}]  if $k=6$, then  $p \equiv 13$, $19$ or $31$  mod $42$.
\end{enumerate}
\end{cor}

\begin{proof}
When  $k = 3$, $4$ or $6$, Lemma~\ref{theta from orders} gives us $\theta(y) = 1, 2, 3$, respectively,   
and in all three cases $\theta(xy)=4$, since $xy$ is the transformation $z \mapsto z+1$.   
Apply the Pythagorean Lemma, by taking $a, b,c$ and $d$ as the four involutions $t,x, xt$ and $ty$ 
respectively, to give $\theta (y) + \theta(xty) + \theta (xy) =4.$  
Note that $\theta(xty) = \theta(xyt)$ because $t^{-1}(xty)t = (t^{-1}xt)yt = xyt$.
Thus we find $\theta(xyt) = \theta(xty) = 4 - \theta(y) - \theta(xy) =   -\theta(y)$, which equals $-1$, $-2$, or  $-3$, respectively, when $k$ is $3$, $4$, or $6$.  
Finally, for the associate to be a januarial we need $xyt$ to have order $(p+1)/2$, 
and so the constraints on $p$ follow from Corollary~\ref{conditions on p}.  
\end{proof}
${}$\\[-40pt] 
 
\section{Examples} \label{examples}

In this section we explore a number of   examples of $k$-januarials guided by Corollary~\ref{necessary conditions}.

We begin with $k=3$. The eight smallest  $p$ satisfying the conditions of Corollary~\ref{necessary conditions} are $13$, $17$, $37$, $53$, $73$, $97$, $113$ and $137$. 
We have already seen how the cases $p=13$ and $p=17$ yield the $3$-januarials depicted 
in Figure~\ref{sphere and torus}.   
The cases where $p$ is $37$, $53$, $73$, $97$ or $137$ all yield januarials.
On the other hand,  
the standard construction (with $x: z \mapsto -1/z$ and $y: z \mapsto (z-1)/z\,$) 
does not yield a januarial in the case $p = 113$, 
since in that case  $xty: z \mapsto (z+1)/z$, which has order $19$, rather than $(113+1)/2 = 57$.

\begin{figure}[ht]
\psfrag{x}{{$\begin{array}{rl} x: & \!\!\! z \mapsto -z \\ y: & \!\!\! z \mapsto (z-1)/z \\ xy: & \!\!\! z \mapsto (z+1)/z \end{array}$}}
\psfrag{i}{\tiny{$\infty$}}
\psfrag{0}{\tiny{$0$}}
\psfrag{1}{\tiny{$1$}}
\psfrag{2}{\tiny{$2$}}
\psfrag{3}{\tiny{$3$}}
\psfrag{4}{\tiny{$4$}}
\psfrag{5}{\tiny{$5$}}
\psfrag{6}{\tiny{$6$}}
\psfrag{7}{\tiny{$7$}}
\psfrag{8}{\tiny{$8$}}
\psfrag{9}{\tiny{$9$}}
\psfrag{10}{\tiny{$10$}}
\psfrag{11}{\tiny{$11$}}
\psfrag{12}{\tiny{$12$}}
\psfrag{13}{\tiny{$13$}}
\psfrag{14}{\tiny{$14$}}
\psfrag{15}{\tiny{$15$}}
\psfrag{16}{\tiny{$16$}}
\psfrag{17}{\tiny{$17$}}
\psfrag{18}{\tiny{$18$}}
\psfrag{19}{\tiny{$19$}}
\psfrag{20}{\tiny{$20$}}
\psfrag{21}{\tiny{$21$}}
\psfrag{22}{\tiny{$22$}}
\psfrag{23}{\tiny{$23$}}
\psfrag{24}{\tiny{$24$}}
\psfrag{25}{\tiny{$25$}}
\psfrag{26}{\tiny{$26$}}
\psfrag{27}{\tiny{$27$}}
\psfrag{28}{\tiny{$28$}}
\psfrag{29}{\tiny{$29$}}
\psfrag{30}{\tiny{$30$}}
\psfrag{31}{\tiny{$31$}}
\psfrag{32}{\tiny{$32$}}
\psfrag{33}{\tiny{$33$}}
\psfrag{34}{\tiny{$34$}}
\psfrag{35}{\tiny{$35$}}
\psfrag{36}{\tiny{$36$}}
\psfrag{37}{\tiny{$37$}}
\psfrag{38}{\tiny{$38$}}
\psfrag{39}{\tiny{$39$}}
\psfrag{40}{\tiny{$40$}}
\psfrag{41}{\tiny{$41$}}
\psfrag{42}{\tiny{$42$}}
\psfrag{43}{\tiny{$43$}}
\psfrag{44}{\tiny{$44$}}
\psfrag{45}{\tiny{$45$}}
\psfrag{46}{\tiny{$46$}}
\psfrag{47}{\tiny{$47$}}
\psfrag{48}{\tiny{$48$}}
\psfrag{49}{\tiny{$49$}}
\psfrag{50}{\tiny{$50$}}
\psfrag{51}{\tiny{$51$}}
\psfrag{52}{\tiny{$52$}}
\psfrag{53}{\tiny{$53$}}
\psfrag{54}{\tiny{$54$}}
\psfrag{55}{\tiny{$55$}}
\psfrag{56}{\tiny{$56$}}
\psfrag{57}{\tiny{$57$}}
\psfrag{58}{\tiny{$58$}}
\psfrag{59}{\tiny{$59$}}
\psfrag{60}{\tiny{$60$}}
\psfrag{61}{\tiny{$61$}}
\psfrag{62}{\tiny{$62$}}
\psfrag{63}{\tiny{$63$}}
\psfrag{64}{\tiny{$64$}}
\psfrag{65}{\tiny{$65$}}
\psfrag{66}{\tiny{$66$}}
\psfrag{67}{\tiny{$67$}}
\psfrag{68}{\tiny{$68$}}
\psfrag{69}{\tiny{$69$}}
\psfrag{70}{\tiny{$70$}}
\psfrag{71}{\tiny{$71$}}
\psfrag{72}{\tiny{$72$}}
\centerline{\epsfig{file=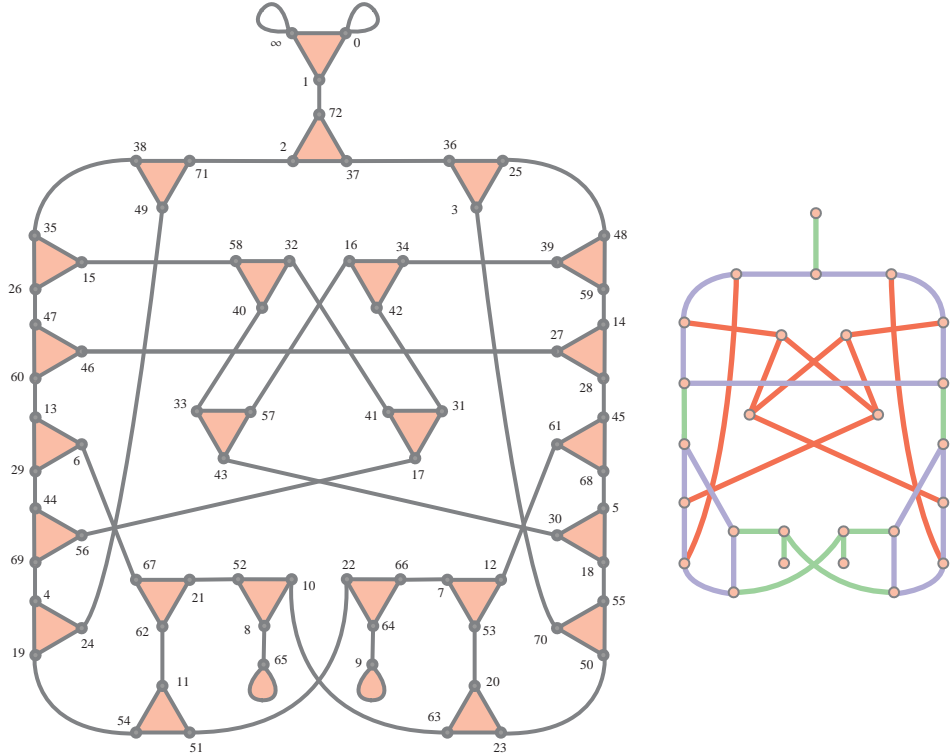}}
\caption{The associate of the standard action of $\PSL(2,73)$ on $\F_{73} \cup \set{\infty}$, 
giving an action of $\Delta(2,3,37)$ via $z \mapsto -z$ and $z \mapsto (z-1)/z$.  
This yields a $3$-januarial of genus $5$ and simple type $(3,2,1)$.} 
\label{PSL73 associate figure}
\end{figure}

The action of $\PSL(2,73)$ on $\F_{73} \cup \set{\infty}$ via $x: z \mapsto -1/z$ 
and $y: z \mapsto (z-1)/z$ gives a coset graph, the associate of which is depicted 
together with its companion graph in Figure~\ref{PSL73 associate figure}.  

The number of $x$-edges which are not loops  and $y$-faces in the associate 
are $36$ and $26$, respectively, so  the genus is $(36 -26)/2=5$ by Lemma~\ref{genus}.  
The type of the januarial is apparent from $\bar{\Gamma}$.  
The blue subgraph (which is the common boundary of $\bar{S_1}$ and $\bar{S_2}$) 
consists of three disjoint simple closed curves, and so $h=3$ by Lemma~\ref{simple type}.  
There are $22$ vertices on $\bar{S_1}$ and $21$ on $\bar{S_2}$; 
there are  $26$ edges on $\bar{S_1}$ (coloured green and blue), 
and  $27$ edges on $\bar{S_2}$ (coloured red and blue); 
and both $\bar{S_1}$ and $\bar{S_2}$ have $1$ face and $3$ holes.   
Hence by Lemma~\ref{finding g_i}, the genera of $\bar{S_1}$ and $\bar{S_2}$ 
are $(2-22+26-3-1)/2 = 1$ and $(2-21+27-3-1)/2 =2$, respectively.  
Accordingly, the $3$-januarial is of type $(3,2,1)$, and this gives an alternative means 
of identifying the genus as  $2 + 1 + 3-1 = 5$, by Lemma~\ref{genus from type}.

We now turn to $4$-januarials.  

For the standard construction (with $x: z \mapsto -1/z$ and $y: z \mapsto (z-1)/z\,$) 
to give a $4$-januarial for $\PSL(2,p)$, Corollary~\ref{necessary conditions} 
requires $p \equiv 17$ or $19$ mod $24$.  
The eight smallest values for the prime $p$ satisfying this condition 
are $17$, $19$, $41$, $43$, $67$, $89$, $113$ and $117$.  
But also if we want $xt$ to lie in $\PSL(2,p)$, 
we need $t: z \mapsto 1/z$ to lie in $\PSL(2,p)$, and then we need $-1$ to be a square 
mod $p$, and so $p \equiv 1$ mod $4$. 
The cases $p = 17$, $89$, $113$ and $117$ all give simple $4$-januarials for $\PSL(2,p)$ 
in this way, while the case $p = 41$ fails, since in that case $xy$ has order $7$ rather than 
the required $(41+1)/2 = 21$. 

On the other hand, if we are happy to construct $4$-januarials for $\PGL(2,p)$ instead 
of $\PSL(2,p)$, we can relax the requirements and allow $x$, $y$ or $t$ to lie 
in $\PGL(2,p) \setminus \PSL(2,p)$. 
When we do that, we get simple $4$-januarials for $\PSL(2,p)$ in the cases 
$p = 19$ and $43$ (but not for $p = 67$). 
We consider the case $p=43$ in more detail below.

\begin{figure}[ht]
\psfrag{t}{{$t: z \mapsto 22/z$}} 
\psfrag{i}{\tiny{$\infty$}}
\psfrag{0}{\tiny{$0$}}
\psfrag{1}{\tiny{$1$}}
\psfrag{2}{\tiny{$2$}}
\psfrag{3}{\tiny{$3$}}
\psfrag{4}{\tiny{$4$}}
\psfrag{5}{\tiny{$5$}}
\psfrag{6}{\tiny{$6$}}
\psfrag{7}{\tiny{$7$}}
\psfrag{8}{\tiny{$8$}}
\psfrag{9}{\tiny{$9$}}
\psfrag{10}{\tiny{$10$}}
\psfrag{11}{\tiny{$11$}}
\psfrag{12}{\tiny{$12$}}
\psfrag{13}{\tiny{$13$}}
\psfrag{14}{\tiny{$14$}}
\psfrag{15}{\tiny{$15$}}
\psfrag{16}{\tiny{$16$}}
\psfrag{17}{\tiny{$17$}}
\psfrag{18}{\tiny{$18$}}
\psfrag{19}{\tiny{$19$}}
\psfrag{20}{\tiny{$20$}}
\psfrag{21}{\tiny{$21$}}
\psfrag{22}{\tiny{$22$}}
\psfrag{23}{\tiny{$23$}}
\psfrag{24}{\tiny{$24$}}
\psfrag{25}{\tiny{$25$}}
\psfrag{26}{\tiny{$26$}}
\psfrag{27}{\tiny{$27$}}
\psfrag{28}{\tiny{$28$}}
\psfrag{29}{\tiny{$29$}}
\psfrag{30}{\tiny{$30$}}
\psfrag{31}{\tiny{$31$}}
\psfrag{32}{\tiny{$32$}}
\psfrag{33}{\tiny{$33$}}
\psfrag{34}{\tiny{$34$}}
\psfrag{35}{\tiny{$35$}}
\psfrag{36}{\tiny{$36$}}
\psfrag{37}{\tiny{$37$}}
\psfrag{38}{\tiny{$38$}}
\psfrag{39}{\tiny{$39$}}
\psfrag{40}{\tiny{$40$}}
\psfrag{41}{\tiny{$41$}}
\psfrag{42}{\tiny{$42$}}
\centerline{\epsfig{file=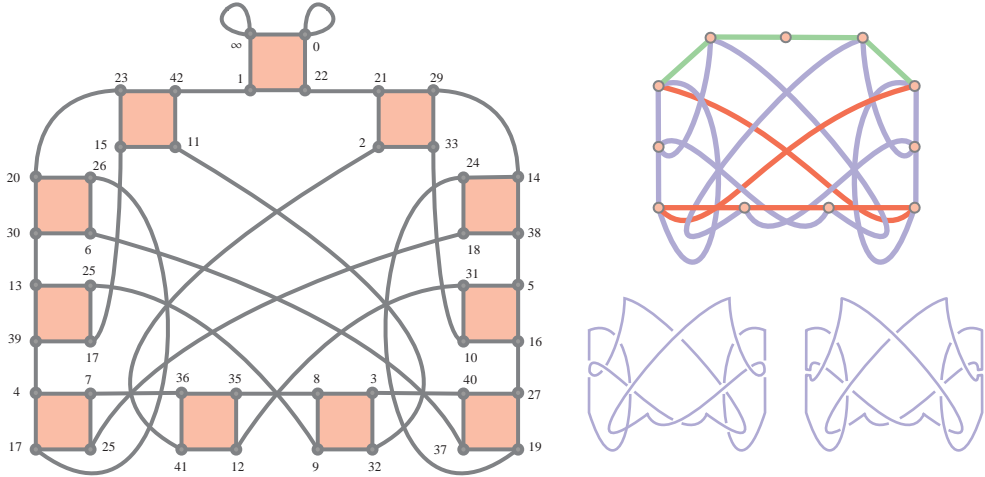}}
\caption{A coset diagram for the action of 
$\Delta(2,4,22)$ 
on $\F_{43} \cup \set{\infty}$ given by $z \mapsto 21z/22$ and $y:  z \mapsto (2z-1)/2z$, 
its companion graph, and two partitions of the subgraph $\mathcal{G}$ into circuits.  
This yields a $4$-januarial of genus $5$ and general type $((4,3),(2,2))$.}   
\label{43_associate}
\end{figure}

Letting $y$ be the transformation $z \mapsto (2z-1)/2z$, we can taken $x$ as $z \mapsto 21/z$ 
and once more get the product $xy$ as the (parabolic) transformation $z \mapsto z+1$, which fixes 
$\infty$ and induces a $43$-cycle on the remaining points. 
The generator $y$ induces a permutation with eleven $4$-cycles and no fixed points, 
while the generator $x$ fixes two points (namely $8$ and $35$) and induces 21 transpositions 
on the remaining points.

 We have not drawn the resulting coset graph, but note that it is reflexible, 
via the transformation $t: z \mapsto 22/z$.  
Its associate, given by the triple $(xt,y,xty)$, shown alongside its companion graph in Figure~\ref{43_associate}, produces a 4-januarial, since 
$xty$ is the transformation $z \mapsto (z+22)/z$, which has two cycles of length $22$. 

The genus is $(21-11)/2 = 5$ by Lemma~\ref{genus}.  
We can use Lemma~\ref{finding h_i} to find $h_1$ and $h_2$.  The partition arising from  
the attaching map $\rho_2$ (following the green and blue edges)  comprises $h_1 =4$ circuits, 
while the partition for $\rho_1$  (following the red and blue edges) comprises $h_2=2$ circuits.  
Both are indicated in the figure.  
By Lemma~\ref{finding g_i}, we have $2 - 2g_1 = 8 -17+4+1$ and $2 - 2g_2 = 11 -16+2+1$, 
and so $g_1 =3$ and $g_2= 2$.


Finally in this section, we consider $6$-januarials.  

In this case, Corollary~\ref{necessary conditions} requires $p \equiv 13$, $19$ or $31$  mod $42$.  
Figure~\ref{PSL31} shows a coset graph for $\PGL(2,31)$.  
Its associate, which yields a januarial $J$, is shown in Figure~\ref{PSL31_associate}  
together with the companion graph.   
By Lemma~\ref{genus},  the genus of $J$ is $(15 - 7)/2 =4$.  
Lemma~\ref{finding h_i} gives $h_1$ and $h_2$:  as shown in the figure, 
we find that $h_1 =1$ circuits comprise the partition of $\mathcal{G}$ arising from 
the attaching map $\rho_2$ (following the green and blue edges), 
and $h_2 =4$ comprise the partition arising from $\rho_1$ (following the red and blue edges).  
Hence by Lemma~\ref{finding g_i}, we find that $2-2g_1=5-13+1+1$ and $2-2g_2 =7-12+4+1$, 
and therefore $g_1=4$ and $g_2=1$.

\begin{figure}[ht]
\psfrag{t}{$t: z \mapsto 1/(3z)$}
\psfrag{i}{\tiny{$\infty$}}
\psfrag{0}{\tiny{$0$}}
\psfrag{1}{\tiny{$1$}}
\psfrag{2}{\tiny{$2$}}
\psfrag{3}{\tiny{$3$}}
\psfrag{4}{\tiny{$4$}}
\psfrag{5}{\tiny{$5$}}
\psfrag{6}{\tiny{$6$}}
\psfrag{7}{\tiny{$7$}}
\psfrag{8}{\tiny{$8$}}
\psfrag{9}{\tiny{$9$}}
\psfrag{10}{\tiny{$10$}}
\psfrag{11}{\tiny{$11$}}
\psfrag{12}{\tiny{$12$}}
\psfrag{13}{\tiny{$13$}}
\psfrag{14}{\tiny{$14$}}
\psfrag{15}{\tiny{$15$}}
\psfrag{16}{\tiny{$16$}}
\psfrag{17}{\tiny{$17$}}
\psfrag{18}{\tiny{$18$}}
\psfrag{19}{\tiny{$19$}}
\psfrag{20}{\tiny{$20$}}
\psfrag{21}{\tiny{$21$}}
\psfrag{22}{\tiny{$22$}}
\psfrag{23}{\tiny{$23$}}
\psfrag{24}{\tiny{$24$}}
\psfrag{25}{\tiny{$25$}}
\psfrag{26}{\tiny{$26$}}
\psfrag{27}{\tiny{$27$}}
\psfrag{28}{\tiny{$28$}}
\psfrag{29}{\tiny{$29$}}
\psfrag{30}{\tiny{$30$}}
\centerline{\epsfig{file=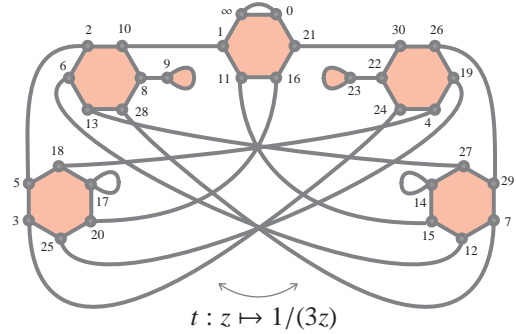}}
\caption{A coset graph for the action of 
$\Delta(2,6,31)$  
on $\F_{31} \cup \set{\infty}$ 
given by $x: z \mapsto 10/z$ and $y:  z \mapsto (z+10)/z$.} 
\label{PSL31}
\end{figure}

\begin{figure}[ht]
\psfrag{i}{\tiny{$\infty$}}
\psfrag{0}{\tiny{$0$}}
\psfrag{1}{\tiny{$1$}}
\psfrag{2}{\tiny{$2$}}
\psfrag{3}{\tiny{$3$}}
\psfrag{4}{\tiny{$4$}}
\psfrag{5}{\tiny{$5$}}
\psfrag{6}{\tiny{$6$}}
\psfrag{7}{\tiny{$7$}}
\psfrag{8}{\tiny{$8$}}
\psfrag{9}{\tiny{$9$}}
\psfrag{10}{\tiny{$10$}}
\psfrag{11}{\tiny{$11$}}
\psfrag{12}{\tiny{$12$}}
\psfrag{13}{\tiny{$13$}}
\psfrag{14}{\tiny{$14$}}
\psfrag{15}{\tiny{$15$}}
\psfrag{16}{\tiny{$16$}}
\psfrag{17}{\tiny{$17$}}
\psfrag{18}{\tiny{$18$}}
\psfrag{19}{\tiny{$19$}}
\psfrag{20}{\tiny{$20$}}
\psfrag{21}{\tiny{$21$}}
\psfrag{22}{\tiny{$22$}}
\psfrag{23}{\tiny{$23$}}
\psfrag{24}{\tiny{$24$}}
\psfrag{25}{\tiny{$25$}}
\psfrag{26}{\tiny{$26$}}
\psfrag{27}{\tiny{$27$}}
\psfrag{28}{\tiny{$28$}}
\psfrag{29}{\tiny{$29$}}
\psfrag{30}{\tiny{$30$}}
\centerline{\epsfig{file=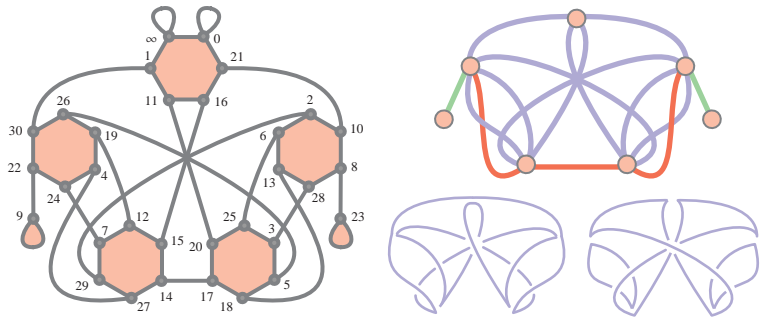}}
\caption{The associate of Figure~\ref{PSL31}, giving an action of 
$\Delta(2,6,16)$ 
on $\F_{31} \cup \set{\infty}$ via $x:   z \mapsto z/30$ and $y:  z \mapsto (z+10)/z$, 
together with its companion graph, and two partitions of the subgraph $\mathcal{G}$ into circuits. 
The resulting $6$-januarial is of genus $4$ and general type $((1,4),(4,1))$.} 
\label{PSL31_associate}
\end{figure}

\bs

\section{Afterword} \label{afterword}

\subsection{Higman's portrait}

Higman delivered the  lectures on which this account is based in the \emph{Higman Room} 
of the Mathematical Institute at Oxford University.   As he spoke, he could see an image of 
himself looking on, from his 1984 portrait by Norman Blamey,  which is reproduced below.  
\begin{figure}[ht] 
${}$\\[+18pt] 
\centerline{\epsfig{file=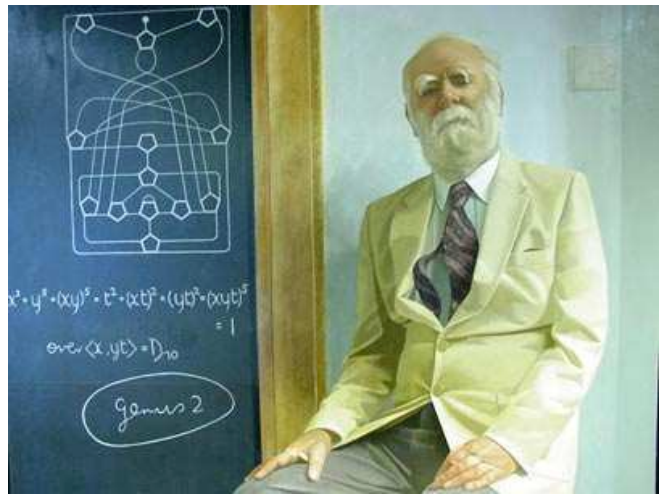,height=65mm}}
\caption{Norman Blamey's 1984 portrait  of Graham Higman} \label{Higman portrait} 
\end{figure} 
 
The portrait shows Higman beside a coset graph for the action of the group $\PSL(2,11)$ on
the cosets of a dihedral subgroup of order $10$ and index $66$.  Equivalently, it gives the 
natural action of $\PSL(2,11)$ on the 66 unordered pairs of points on the projective 
line over a field of order $11$.  The two generators $x$ and $y$ satisfy the relations 
$x^2 = y^5 = (xy)^5 = 1$, but also the diagram is reflexible about a vertical axis of symmetry, 
and the reflection is achievable by conjugation by an involution $t$ in the same group.  

In fact $x$ and $yt$ may be taken as involutory generators of the stabilizer of the 
pair $\set{0,\infty}$, such as  $z \mapsto -1/z$  and  $z \mapsto 2/z$, and $t$ as the 
transformation $z \mapsto (z+1)/(z-1)$.  These choices make $y$ the transformation 
$z \mapsto (z+2)/(2-z)$.   The three generators $x$, $y$ and $t$ then satisfy the 
relations written on the blackboard in the portrait, namely 
$$
 x^2 \ = \  y^5 \ = \  (xy)^5  \ =  \ t^2  \ = \  (xt)^2 \  = \  (yt)^2 \ = \  (xyt)^5 \  = \ 1, 
 $$
 which are the defining relations for the group $G^{5,5,5}$ in the notation of 
 Coxeter \cite{Cox-Gmnp}.  
 Hence in particular, $G^{5,5,5}$ is isomorphic to $\PSL(2,11)$. 
 
The diagram does not give a januarial, but rather a $13$-face map.   
The associated  surface has genus $2$, since there are $13$ pentagons corresponding to the 
$5$-cycles of $\langle y \rangle$, and $28$ edges between distinct pairs of such pentagons  
(from transpositions of $x$), and $13$ faces coming from the 
$5$-cycles of $\langle xy \rangle$, giving Euler characteristic $13-28+13=-2$.   
The isomorphism with $G^{5,5,5}$ also makes $\PSL(2,11)$ the automorphism group of 
a {\em regular map} of type $\{5,5\}_5$ on a non-orientable surface of Euler characteristic $-33$ 
(see \cite{CoxMos}), and hence also the automorphism group of a regular $3$-polytope 
of type $[5,5]$.

\subsection{Other sources of januarials}

Many januarials can also be constructed from groups 
other than $\PSL(2,q)$ and $\PGL(2,q)$.  
For example, the alternating group $\textup{Alt}(16)$ is generated by elements
$\,x = (2, 4)(3, 7)(6,$ $10)(8, 16)(9, 13)(11, 14)\,$ and 
$\,y = (1, 2, 3)(4, 5, 6)(7, 8, 9)$ $(10, 11, 12)(13, 14, 15), \, $
with product
$\, xy = (1, 2, 5, 6, 11, 15, 13, 7)(3, 8, 16, 9, 14, 12, 10, 4),\, $ 
which has two cycles of length $8$. 
The resulting coset diagram is shown in Figure~\ref{A_16}.

\begin{figure}[ht]
\psfrag{1}{\tiny{$1$}}
\psfrag{2}{\tiny{$2$}}
\psfrag{3}{\tiny{$3$}}
\psfrag{4}{\tiny{$4$}}
\psfrag{5}{\tiny{$5$}}
\psfrag{6}{\tiny{$6$}}
\psfrag{7}{\tiny{$7$}}
\psfrag{8}{\tiny{$8$}}
\psfrag{9}{\tiny{$9$}}
\psfrag{10}{\tiny{$10$}}
\psfrag{11}{\tiny{$11$}}
\psfrag{12}{\tiny{$12$}}
\psfrag{13}{\tiny{$13$}}
\psfrag{14}{\tiny{$14$}}
\psfrag{15}{\tiny{$15$}}
\psfrag{16}{\tiny{$16$}}
\centerline{\epsfig{file=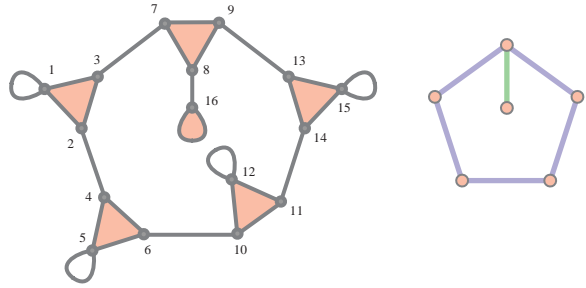}} 
\caption{A coset graph for an action of $\textup{Alt}(16)$ of degree $16$, together with its companion graph.  
This gives a $3$-januarial of genus $0$ and simple type $( 1, 0, 0)$.} 
\label{A_16}
\end{figure}

\medskip 
 


Other examples are obtainable from the groups $\PSL(2,q)$  and 
$\PGL(2,q)$ without taking the approach that we did in Section~\ref{finding} 
which had $xy$ as the transformation $z \mapsto z+1$.  An example is 
given in Figure~\ref{11}.

\begin{figure}[ht]
\psfrag{i}{\tiny{$\infty$}}
\psfrag{0}{\tiny{$0$}}
\psfrag{1}{\tiny{$1$}}
\psfrag{2}{\tiny{$2$}}
\psfrag{3}{\tiny{$3$}}
\psfrag{4}{\tiny{$4$}}
\psfrag{5}{\tiny{$5$}}
\psfrag{6}{\tiny{$6$}}
\psfrag{7}{\tiny{$7$}}
\psfrag{8}{\tiny{$8$}}
\psfrag{9}{\tiny{$9$}}
\psfrag{10}{\tiny{$10$}}
\psfrag{11}{\tiny{$11$}}
\psfrag{12}{\tiny{$12$}}
\psfrag{13}{\tiny{$13$}}
\psfrag{14}{\tiny{$14$}}
\psfrag{15}{\tiny{$15$}}
\psfrag{16}{\tiny{$16$}}
\centerline{\epsfig{file=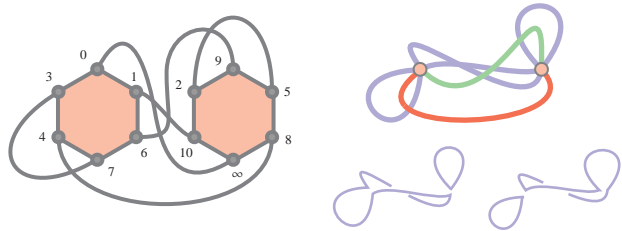}} 
\caption{A coset graph for an action of $\textup{PSL}(2,11)$ 
on $\mathbb{F}_{11} \cup \set{\infty}$, via 
$x : z \mapsto -1/z$ and  $y : z \mapsto  (8z-8)/(z+1)$, 
together with its companion graph, and two partitions 
of the subgraph $\mathcal{G}$.  
This results in a $6$-januarial of genus $1$ and 
general type $( (2,1), (2,1))$.} \label{11}
\end{figure}

\end{document}